\documentclass{article}

\usepackage{amsmath,amsthm,amssymb,amsfonts}

\usepackage{verbatim}
\usepackage{graphicx}

\usepackage{stmaryrd} %for \leftrightarroweq
\usepackage{hyperref}

\usepackage[colorinlistoftodos]{todonotes}

\newtheorem{thm}{Theorem}

\newtheorem{prop}[thm]{Proposition}

\newtheorem{assert}[thm]{Assertion}

\newtheorem{remarks}[thm]{Remark}

\newtheorem{definition}[thm]{Definition}

\newtheorem{exl}[thm]{Example}

\numberwithin{thm}{section}

\newcommand{\adj}{\leftrightarrow}

\newcommand{\adjeq}{\leftrightarroweq}

\def\Z{{\mathbb Z}}
\def\N{{\mathbb N}}

\begin{document}
\title{Remarks on Fixed Point Assertions in Digital Topology, 11}
\author{Laurence Boxer
\thanks{Department of Computer and Information Sciences, Niagara University, NY 14109, USA
and  \newline
Department of Computer Science and Engineering, State University of New York at Buffalo \newline
email: boxer@niagara.edu
\newline
ORCID: 0000-0001-7905-9643
}
}

\date{ }
\maketitle

\begin{abstract} The topic of fixed points in digital metric spaces has drawn yet more
publications with assertions that are incorrect, incorrectly proven, trivial, or
incoherently stated. We discuss publications with 
bad assertions concerning fixed points of self-functions on digital images, as in some of
our previous papers.

MSC: 54H25

Key words and phrases: digital topology, digital image,
fixed point, digital metric space
\end{abstract}

\section{Introduction}
{\bf Fixed points in digital topology} - this is a topic that has inspired some beautiful
results. It has also been at the center of many assertions that are incorrect, incorrectly
proven, trivial, or not presented coherently. In the current paper, we continue the work
of~\cite{BxSt19, Bx19, Bx19-3, Bx20, Bx22, BxBad6, BxBad7, BxBad8, BxBad9, BxBad10}
as we discuss several other papers with assertions that merit at least one of the 
descriptions in the previous sentence. Papers drawing our disapproval in
the current work have all come to our attention since acceptance for publication
of~\cite{BxBad10}.

\section{Preliminaries}

\subsection{Adjacencies, 
%connectedness, 
continuity, fixed point}

Much of the material in this section is quoted or
paraphrased from~\cite{BxBad10}.

In a digital image $(X,\kappa)$, if
$x,y \in X$, we use the notation
$x \adj_{\kappa}y$ to
mean $x$ and $y$ are $\kappa$-adjacent; we may write
$x \adj y$ when $\kappa$ can be understood. 
We write $x \adjeq_{\kappa}y$, or $x \adjeq y$
when $\kappa$ can be understood, to
mean 
$x \adj_{\kappa}y$ or $x=y$.

The most commonly used adjacencies in the study of digital images 
are the $c_u$ adjacencies. These are defined as follows.
\begin{definition}
\label{cu-adj-Def}
Let $X \subset \Z^n$. Let $u \in \Z$, $1 \le u \le n$. Let 
$x=(x_1, \ldots, x_n),~y=(y_1,\ldots,y_n) \in X$. Then $x \adj_{c_u} y$ if 
\begin{itemize}
    \item $x \neq y$,
    \item for at most $u$ distinct indices~$i$,
    $|x_i - y_i| = 1$, and
    \item for all indices $j$ such that $|x_j - y_j| \neq 1$ we have $x_j=y_j$.
\end{itemize}
\end{definition}

\begin{definition}
\label{path}
{\rm (See \cite{Khalimsky})} 
    Let $(X,\kappa)$ be a digital image. Let
    $x,y \in X$. Suppose there is a set
    $P = \{x_i\}_{i=0}^n \subset X$ such that
$x=x_0$, $x_i \adj_{\kappa} x_{i+1}$ for
$0 \le i < n$, and $x_n=y$. Then $P$ is a
{\em $\kappa$-path} (or just a {\em path}
when $\kappa$ is understood) in $X$ from $x$ to $y$,
and $n$ is the {\em length} of this path.
\end{definition}

\begin{definition}
{\rm \cite{Rosenfeld}}
A digital image $(X,\kappa)$ is
{\em $\kappa$-connected}, or just {\em connected} when
$\kappa$ is understood, if given $x,y \in X$ there
is a $\kappa$-path in $X$ from $x$ to $y$. The {\rm $\kappa$-component of~$x$ in~$X$} is the
maximal $\kappa$-connected subset
of~$X$ containing~$x$.
\end{definition}

\begin{definition}
{\rm \cite{Rosenfeld, Bx99}}
Let $(X,\kappa)$ and $(Y,\lambda)$ be digital
images. A function $f: X \to Y$ is 
{\em $(\kappa,\lambda)$-continuous}, or
{\em $\kappa$-continuous} if $(X,\kappa)=(Y,\lambda)$, or
{\em digitally continuous} when $\kappa$ and
$\lambda$ are understood, if for every
$\kappa$-connected subset $X'$ of $X$,
$f(X')$ is a $\lambda$-connected subset of $Y$.
\end{definition}

\begin{thm}
{\rm \cite{Bx99}}
A function $f: X \to Y$ between digital images
$(X,\kappa)$ and $(Y,\lambda)$ is
$(\kappa,\lambda)$-continuous if and only if for
every $x,y \in X$, if $x \adj_{\kappa} y$ then
$f(x) \adjeq_{\lambda} f(y)$.
\end{thm}

\begin{comment}
\begin{remarks}
    For $x,y \in X$, $P = \{x_i\}_{i=0}^n \subset X$
is a $\kappa$-path from $x$ to $y$ if and only if
$f: [0,n]_{\Z} \to X$, given by $f(i)=x_i$, is
$(c_1,\kappa)$-continuous. Therefore, we may also
call such a function $f$ a {\em ($\kappa$-)path}
in $X$ from $x$ to $y$.
\end{remarks}
\end{comment}

\begin{comment}
\begin{thm}
\label{composition}
{\rm \cite{Bx99}}
Let $f: (X, \kappa) \to (Y, \lambda)$ and
$g: (Y, \lambda) \to (Z, \mu)$ be continuous 
functions between digital images. Then
$g \circ f: (X, \kappa) \to (Z, \mu)$ is continuous.
\end{thm}
\end{comment}

%We use $\id_X$ to denote the identity function on $X$, and 
%$C(X,\kappa)$ for the set of functions $f: X \to X$ that are $\kappa$-continuous.

A {\em fixed point} of a function $f: X \to X$ 
is a point $x \in X$ such that $f(x) = x$. 
%We denote by $\Fix(f)$ the set of fixed points of $f: X \to X$.

\begin{comment}
Let $X = \Pi_{i=1}^n X_i$. The 
{\em projection to the
$j^{th}$ coordinate}
function $p_j: X \to X_j$
is the function defined
for $x = (x_1, \ldots, x_n) \in X$, $x_i \in X_i$, by $p_j(x) = x_j$.
\end{comment}

As a convenience, if $x$ is a point in the domain of a function $f$, we will often
abbreviate ``$f(x)$" as ``$fx$".
%Also, if $f: X \to X$, we use
%``$f^n$" for the $n$-fold composition
%\[ f^n = \overbrace{f \circ \ldots \circ f}^{n}.
%\]

\subsection{Digital metric spaces}
\label{DigMetSp}
A {\em digital metric space}~\cite{EgeKaraca-Ban} is a triple
$(X,d,\kappa)$, where $(X,\kappa)$ is a digital image and $d$ is a metric on $X$. The
metric is usually taken to be the Euclidean
metric or some other $\ell_p$ metric; 
alternately, $d$ might be taken to be the
shortest path metric. These are defined
as follows.
\begin{itemize}
    \item Given 
          $x = (x_1, \ldots, x_n) \in \Z^n$,
          $y = (y_1, \ldots, y_n) \in \Z^n$,
          $p > 0$, $d$ is the $\ell_p$ metric
          if \[ d(x,y) =
          \left ( \sum_{i=1}^n
          \mid x_i - y_i \mid ^ p
          \right ) ^ {1/p}. \]
          Note the special cases: if $p=1$ we
          have the {\em Manhattan metric}; if
          $p=2$ we have the 
          {\em Euclidean metric}.
    \item \cite{ChartTian} If $(X,\kappa)$ is a 
          connected digital image, 
          $d$ is the {\em shortest path metric}
          if for $x,y \in X$, $d(x,y)$ is the 
          length of a shortest
          $\kappa$-path in $X$ from $x$ to $y$.
\end{itemize}

%\begin{definition}
%    \label{nbd}
%    Given $x \in X$, where $(X,\kappa)$ is a digital
%    image, the {\rm neighborhood of~$x$ of radius} $r \in \N^*$ is
%    \[ N_{\kappa}(x,r) = \{ y \in X \mid d(x,y) \le r \}
%    \]
%    where $d$ is the shortest path
%    metric for the $\kappa$-component of~$x$ in~$X$.
%\end{definition}

We say a metric space $(X,d)$ is {\em uniformly discrete}
if there exists $\varepsilon > 0$ such that
$x,y \in X$ and $d(x,y) < \varepsilon$ implies $x=y$.

\begin{remarks}
\label{unifDiscrete}
If $X$ is finite or  
\begin{itemize}
\item {\rm \cite{Bx19-3}}
$d$ is an $\ell_p$ metric, or
\item $(X,\kappa)$ is connected and $d$ is 
the shortest path metric,
\end{itemize}
then $(X,d)$ is uniformly discrete.

For an example of a digital metric space
that is not uniformly discrete, see
Example~2.10 of~{\rm \cite{Bx20}}.
\end{remarks}

We say a sequence $\{x_n\}_{n=0}^{\infty}$ is 
{\em eventually constant} if for some $m>0$, 
$n>m$ implies $x_n=x_m$.
The notions of convergent sequence and complete digital metric space are often trivial, 
e.g., if the digital image is uniformly 
discrete, as noted in the following, a minor 
generalization of results 
of~\cite{HanBan,BxSt19}.

\begin{prop}
\label{eventuallyConst}
{\rm \cite{Bx20}}
If $(X,d)$ is a uniformly discrete metric space,
then any Cauchy sequence in $X$
is eventually constant, and $(X,d)$ is a complete metric space.
\end{prop}

\subsection{On fixed points}
A function $f: X \to X$ has a {\em fixed point} $x_0 \in X$ if
$fx_0 = x_0$. If $X$ is a topological space or if $(X,\kappa)$ is
a digital image, $X$ or, respectively, $(X,\kappa)$ has the
{\em fixed point property} (FPP) if for every continuous
(respectively, $\kappa$-continuous) $f: X \to X$ has a fixed point.
But the FPP turns out to be trivial in digital topology, as shown be the following.

\begin{thm}
\label{BEKLLthm}
    {\rm \cite{BEKLL}} A digital image $(X,\kappa)$ has the FPP if
    and only if $\#X = 1$.
\end{thm}

\section{\cite{DineshEtal}'s ``novel approach"}
One of the flaws of~\cite{DineshEtal} is its extensive use of
unnecessarily complex notation. Where possible, we simplify the
notation of~\cite{DineshEtal}.

\subsection{The introduction}
We mention many errors that appear in the Introduction 
of~\cite{DineshEtal}. Among these are attributions of
definitions and theorems to papers that used these, rather than to
papers in which they first appeared.

Definition 2.1 of~\cite{DineshEtal} - $\Z^q$ - is both incorrectly
attributed and incorrectly stated. This definition is attributed
in~\cite{DineshEtal} to~\cite{DolhareNal} rather than assuming the
reader will know it is classical. The definition stated is
\[ Z^q = \{ (a_1,a_2,a_3, \ldots, a_n) \mid  a_i \in \Z, 1 \le i \le q \}
\]
rather than
\[ Z^q = \{ (a_1,a_2,a_3, \ldots, a_q) \mid  a_i \in \Z, 1 \le i \le q \}.
\]

Definition 2.2 of~\cite{DineshEtal} - the~$c_u$ (also called~$k_u$ and $\ell_u$
in the literature) adjacency  - is both incorrectly
attributed and incorrectly stated. The given attribution is
to~\cite{EgeKaraca-Ban,DolhareNal} but should be to~\cite{BxHtpyProps}.
The definition stated in~\cite{DineshEtal} claims that points 
$x=(x_1, \ldots, x_q)$ and $y=(y_1, \ldots, y_q)$ are adjacent in this adjacency if at most~$u$ indices~$i$ satisfy
$|x_i-y_i| = 1$ and for all other indices~$j$, $|x_i-y_i| \neq 1$.
The latter inequality should be $|x_j-y_j| =0$.

Other incorrect attributions:
\[
  \begin{array}{lll}
     \underline{Entry} & \underline{\cite{DineshEtal}~attribution} & \underline{Better~attribution}  \\
      Def.~2.3~(dig.~img.) & \cite{EgeKaraca-Ban,DolhareNal} & 
           \cite{Rosenf79,Rosenfeld} \\
       Def.~2.4~(fixed pt.) & \cite{EgeKaraca-Ban,DolhareNal} & 
           classic    \\
      Def.~2.4~(FPP) & \cite{EgeKaraca-Ban,DolhareNal} & \cite{Rosenfeld} \\
      Thm.~2.5~(dig.~Banach~princ.) & \cite{EgeKaraca-Ban} & \cite{EgeKaraca-Ban};~ corrected~ proof~\cite{BxBad10}
  \end{array}
\]

\begin{figure}
    \centering
    \includegraphics[width=0.9\linewidth]{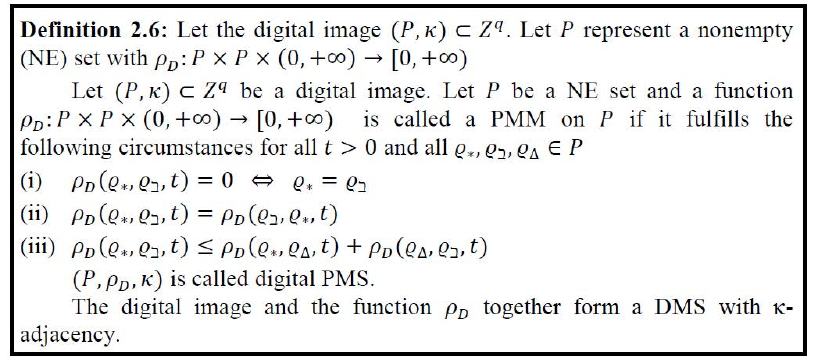}
    \caption{Definition 2.6 of \cite{DineshEtal}}
    \label{fig:DineshEtal2.6}
\end{figure}

Definition 2.6 of~\cite{DineshEtal} (see Figure~\ref{fig:DineshEtal2.6}) defines a function that has the
properties of a metric, but also has a third 
parameter $t > 0$ that is not used for any purpose in the paper.
Thus by taking~$t$ to be a positive constant, we see 
that we can omit this parameter
and replace this function by an actual metric.

Definition 2.8 of~\cite{DineshEtal} fails to clarify in
its statement that the
continuity it uses is of the classical $\varepsilon - \delta$
type. Example~4.1 of~\cite{BxSt19} shows that the classical 
$\varepsilon - \delta$ continuity does not imply digital continuity.

Lemma 2.9 of~\cite{DineshEtal} is unattributed. It should be 
attributed to Theorem~3.7 of~\cite{EgeKaraca-Ban}
and to Theorem~3.1 of~\cite{BxBad10}; the former does not have a
correct proof, and the latter does.

Lemma 2.10 of~\cite{DineshEtal} essentially duplicates 
Lemma 2.9 of~\cite{DineshEtal} and is also unattributed.

\subsection{\cite{DineshEtal}'s ``Theorem" 3.1}
\begin{figure}
    \centering
    \includegraphics[width=0.9\linewidth]{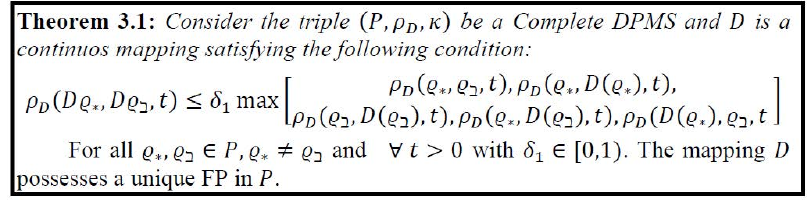}
    \caption{``Theorem" 3.1 of \cite{DineshEtal}}
    \label{fig:DinesheEtal3.1}
\end{figure}

''Theorem" 3.1 of \cite{DineshEtal} (see Figure~\ref{fig:DinesheEtal3.1})
is restated in this section in a much simpler and correct form, and is 
given a proof that is clearer and correct. 

The argument offered as proof in~\cite{DineshEtal} of this assertion
depends on a constant~$\delta_1 \in [0,1)$. From this, another constant,
here denoted~${\cal L}$~as we don't see how to duplicate in \LaTeX~the symbol used
in~\cite{DineshEtal}, is derived via
\[ {\cal L} = \frac{\delta_1}{1 - \delta_1}.
\]
The authors claim ${\cal L}$ is less than 1. But since $\delta_1$
can be a member of~$[1/2, 1)$, the claim is not true. The consequences of this
error propagate through the ``proof" argument, so
''Theorem" 3.1 of \cite{DineshEtal} is unproven.

A self-mapping on a digital image that satisfies the inequality of
Figure~\ref{fig:DinesheEtal3.1} is a 
{\em quasi-contraction} \cite{Ciric}. Fixed point results for
quasi-contractions on digital metric spaces, using the shortest-path
metric, have been obtained in~\cite{GH,BxBad7}.

We modify the statement of ''Theorem" 3.1 of \cite{DineshEtal}
as follows, and give a proof. We make the following changes from
''Theorem" 3.1 of \cite{DineshEtal}:
\begin{itemize}
    \item We omit the continuity and completeness assumptions.
    \item We add the assumption of uniform discreteness.
    \item We take $1/2$, not $1$, as the upper bound for the constant on the right side
          of~(\ref{quasicontractionIneq}) below.
    \item We simplify notation.
\end{itemize}

\begin{thm}
    \label{DineshEtal3.1corrected}
    Let $(X,d,\kappa)$ be a digital metric space, where~$d$ is
    uniformly discrete. Let
    $f: X \to X$ be a function and let $c$ be a constant,
    $0 \le c < 1/2$, such that for all $x,y \in X$,
    \begin{equation}
    \label{quasicontractionIneq}
        d(fx,fy) \le c \cdot \max \{
     d(x,y),~ d(x,fx),~d(y,fy),~d(x,fy),~d(y,fx) \}.
    \end{equation}
    Then $f$ has a unique fixed point in $X$.
\end{thm}

\begin{proof}
    We use ideas of~\cite{DineshEtal}. Let $x_0 \in X$ and
    $x_{i+1} = fx_i$, $i \ge 0$. For all~$i$,
    \[ d(x_{i+1},x_{i+2}) = d(fx_i, fx_{i+1}) \]
    \[\le c \cdot \max \{ d(x_i,x_{i+1}),~d(x_i,fx_i),~
       d(x_{i+1}, fx_{i+1}),~d(x_i,fx_{i+1}),~d(x_{i+1}, fx_i) \}
    \]
    \[ =  c \cdot \max \{ d(x_i,x_{i+1}),~d(x_i, x_{i+1}),
         ~d(x_{i+1},x_{i+2}),~d(x_i, x_{i+2}),~d(x_{i+1},x_{i+1}) \}
    \]
    \[ \le c \cdot \max \{ d(x_i,x_{i+1}),~d(x_{i+1}, x_{i+2}),
        ~d(x_i,x_{i+1}) + d(x_{i+1}, x_{i+2}), 0 \}
    \]
    \[ = c \cdot [d(x_i,x_{i+1}) + d(x_{i+1},x_{i+2}) ].
    \]
    Therefore, $(1-c) \cdot d(x_{i+1},x_{i+2}) \le c \cdot d(x_i,x_{i+1})$, or
    \[ d(x_{i+1},x_{i+2})  \le {\cal L} \cdot d(x_i,x_{i+1}),
    \]
    where ${\cal L} = \frac{c}{1-c} < 1$.
    An easy induction yields that
    $d(x_i,x_{i+m}) \le {\cal L}^{i+m} d(x_0,x_1)$. It follows
     from Proposition~\ref{eventuallyConst} that there exists~$N$ 
     such that $i>N$ implies $x_i = x_N$. Thus, $x_N$ is a 
     fixed point of~$f$.

To show the uniqueness of this fixed point, suppose $x'$
    is a fixed point of~$f$. Then
    \[ d(x_N,x') = d(fx_N, fx') \le \]
    \[ c \cdot \max \{ d(x_N,x'),~d(x_N,fx_N),~d(x',fx'), 
    ~d(x_N,fx'),~d(fx_N,x') \} =
    \]
     \[ c \cdot \max \{ d(x_N,x'),~ 0,~0,
       ~d(x_N,x'),~d(x_N,x') \} = c \cdot d(x_N,x').
    \]
    Since $0 \le c < 1/2$, we have  $d(x_N,x') = 0$, i.e., $x_N = x'$.
\end{proof}

\subsection{\cite{DineshEtal}'s Theorem 3.2}
\begin{figure}
    \centering
    \includegraphics[width=0.9\linewidth]{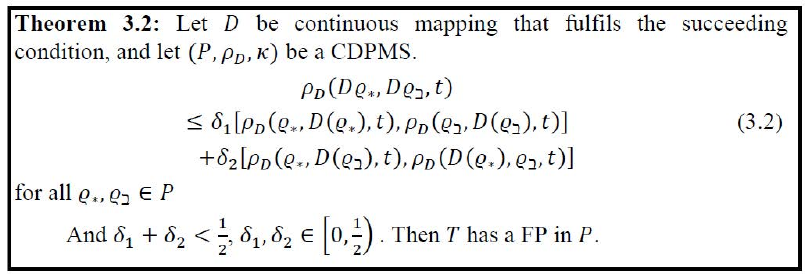}
    \caption{\cite{DineshEtal}'s ``Theorem" 3.2}
    \label{fig:DineshEtal3.2}
\end{figure}

Figure~\ref{fig:DineshEtal3.2} shows \cite{DineshEtal}'s 
Theorem 3.2. We feel our rewriting is
desirable because the flaws and difficult notation of~\cite{DineshEtal}'s
version make both the assertion and its proof difficult to understand.
The assertion is flawed as follows:
\begin{itemize}
    \item It is not clear if the continuity hypothesized is of
          the $\varepsilon - \delta$ variety, or the digital
          variety; however, we will show that neither is required.
    \item The right side of the inequality marked as ``(3.2)" is
        undefined. It appears that there are operators missing -
        apparently, the ``+" in each term according to the argument
        given as proof.
    \item In the conclusion, ``$T$" is apparently meant to be ``$D$".
\end{itemize}

We correct the statement of \cite{DineshEtal}'s 
``Theorem"~3.2. Changes made in the statement of 
the assertion:
\begin{itemize}
    \item We add the assumption that the metric, which we denote as~$d$, is
           uniformly discrete.
    \item We omit the unnecessary assumptions of continuity and completeness
          (the ``C" in ``CDPMS").
    \item We insert the missing ``+" operators.
    \item We use only one symbol for the function with which we are concerned.
    \item We simplify notation.
\end{itemize}
With these changes,
Theorem 3.2 of~\cite{DineshEtal} can be rewritten from the version
shown in~Figure~\ref{fig:DineshEtal3.2} to the following.

\begin{thm}
    \label{DineshEtal3.2corrected}
    Let $(X,d,\kappa)$ be a digital image. Let $f: X \to X$ 
    satisfy, for all $x,y \in X$ and
    some nonnegative constants
    $a,b$ such that $a+b < 1/2$,
    \begin{equation} 
    \label{DineshEtal3.2Ineq}
    d(fx,fy) \le a \cdot [d(x,fx) + d(y,fy)] +
       b \cdot [d(x,fy) + d(y, fx)].
    \end{equation}
    Then~$f$ has a fixed point in~$X$.
\end{thm}

\begin{proof}
Our argument is largely that of~\cite{DineshEtal}.
Using our change of notation, let $x_0 \in X$ and let
a sequence $\{x_i\}_{i=1}^{\infty} \subset X$, where $x_{i+1} = fx_i$, be established.
From the inequality of the assertion, we obtain
  \[ d(x_{i+1}, x_{i+2}) = d(fx_i, fx_{i+1}) \le  
  \]
  \[ a \cdot [d(x_i,fx_{i}) + d(x_{i+1}, fx_{i+1})] +
  b \cdot [d(x_i, fx_{i+1}) + d(x_{i+1},fx_i)] =
  \]
  \[ a \cdot [d(x_i,x_{i+1}) + d(x_{i+1}, x_{i+2})] +
  b \cdot [d(x_i, x_{i+2}) + d(x_{i+1}, x_{i+1})] \le
  \]
 (using the Triangle Inequality)
 \[ a \cdot [d(x_i, x_{i+1}) + d(x_{i+1}, x_{i+2})] + b \cdot [d(x_i, x_{i+1}) + d(x_{i +1}, x_{i+2}) + 0].
 \]
Thus
\[ (1 - (a+b)) d(x_{i+1}, x_{i+2}) \le (a+b) d(x_i, x_{i+1}).
\]
Since $a+b < 1/2$, we have
\[
    0 \le \frac{a+b}{1 - (a+b)} < 1.
\]
Let $c =\frac{a+b}{1 - (a+b)}$. Then $0 \le c < 1$ and
\[  d(x_{i+1}, x_{i+2}) \le c \cdot d(x_i, x_{i+1}).
\]
An easy induction yields that for all positive integers $n,m$,
\[  d(x_n, x_{n+m}) \le  c^{m+n} \cdot d(x_0, x_1)
\to_{m \to \infty} 0.
\]
Thus, $\{x_n\}_{n=0}^{\infty}$ is a Cauchy sequence. By 
Proposition~\ref{eventuallyConst}, there exists $x_N \in X$ such that
for almost all~$i$, $x_i=x_N$. Thus, $x_N$ is a fixed point of~$f$.

To show the uniqueness of~$x_N$ as a fixed point, suppose $x'$ is a fixed
point of~$f$. Then
\[ d(x_N, x') = d(fx_N, fx') \le
\]
\[ a \cdot [d(x_N, fx_N) + d(x', fx')] + b[d(x_N, fx') + d(fx_N, x')] =
\]
\[ a \cdot (0+0) + b \cdot [d(x_N, x') + d(x_N, x')] = 2b \cdot d(x_N, x').
\]
Since $2b < 1$, we have
$d(x_N, x') = 0$, i.e., $x_N = x'$.
\end{proof}

\subsection{\cite{DineshEtal}'s ``Theorem" 3.3}
\begin{figure}
    \centering
    \includegraphics[width=0.9\linewidth]{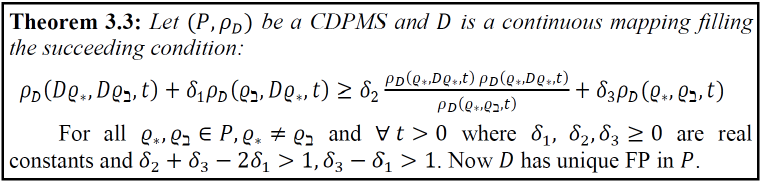}
    \caption{``Theorem" 3.3 of \cite{DineshEtal}}
    \label{fig:DineshEtal3.3}
\end{figure}

``Theorem" 3.3 of \cite{DineshEtal} is shown in 
Figure~\ref{fig:DineshEtal3.3}. As written, the assertion is false, as
shown by the following, stated in simpler notation.

\begin{exl}
    Let $X = \{2^n \mid n \in \N \}$. Let $d(x,y) = |x - y|$. 
    Let $f: X \to X$ be defined by $fx = 2x$. Then for
    any adjacency $\kappa$ on~$X$, $(X,d,\kappa)$ is a uniformly discrete
    digital metric space that satisfies the hypotheses of the assertion
    shown in the equivalent notation of Figure~\ref{fig:DineshEtal3.3}.
    However,~$f$ has no fixed point.
\end{exl}

\begin{proof}
    It is elementary that $(X,d,\kappa)$ is a uniformly discrete
    digital metric space, hence is complete, and that~$f$ has no
    fixed point. Let 
    \[\delta_1 = \delta_2 = 0, ~~~ \delta_3 = 1.5.
    \]
    These values satisfy the restrictions on $\delta_1, \delta_2, \delta_3$ shown
    in~Figure~\ref{fig:DineshEtal3.3} and reduce the distance relation to
    \begin{equation}
        \label{DineshEtal3.3relation}
        d(fx,fy) \ge 1.5 \cdot d(x,y)
    \end{equation}
    Then we have
    \[ d(f(2^i), f(2^j)) = |2^{i+1} - 2^{j+1}| = 2 \cdot d(2^i, 2^j) \ge 1.5 \cdot d(2^i, 2^j),
    \]
 so $f$ satisfies~(\ref{DineshEtal3.3relation}).
\end{proof}

To obtain a true assertion from 
Figure~\ref{fig:DineshEtal3.3}, we make the following changes:
\begin{itemize}
    \item We assume uniform discreteness.
    \item In the 3rd line of Figure~\ref{fig:DineshEtal3.3}, 
   we use ``$\le$" in place of ``$\ge$".
    \item We use ``for some constants" rather than 
    "for all constants".
    \item We use ``$b+c < 1$ instead of ``$b+c-2a > 1$".
    \item We use simpler notation.
    \item We further show that the fixed point is unique.
\end{itemize} 
The assertion can then be restated as follows.

\begin{prop}
    \label{DineshEtal3.3}
    Let $(X,d, \kappa)$ be a uniformly discrete digital metric space.
    Let $f: X \to X$ be such that for some nonnegative constants $a,b,c$ such that 
    $b+c<1$, and for all $x,y \in X$ such that $x \neq y$ we have
    \[ d(fx,fy) + a \cdot d(y,fx) \le b \cdot \frac{d(x,fx) \cdot d(x,fx)}{d(x,y)}
                + c \cdot d(x,y).
    \]
    Then $f$ has a unique fixed point.
\end{prop}

\begin{proof}
    Let $x_0 \in X$ and let $x_{i+1} = fx_i$ for $i \ge 0$. 
    If $x_{i+1} = x_i$ then $x_i$ is a fixed point. Otherwise,
    \[ d(fx_i, fx_{i+1}) + a \cdot d(x_{i+1}, fx_i) \le b \cdot 
       \frac{d(x_i,fx_{i}) \cdot d(x_i,fx_{i})}{d(x_i,x_{i+1})}
                + c \cdot d(x_i,x_{i+1}).
    \]
    This simplifies as
    \[ d(x_{i+1}, x_{i+2}) + a \cdot 0 \le b \cdot 
    \frac{[d(x_i,x_{i+1})]^2}{d(x_i,x_{i+1})} + c \cdot d(x_i,x_{i+1})
    \]
    or
    \begin{equation} 
    \label{DineshEtal3.3correctedIneq}
    d(x_{i+1}, x_{i+2}) \le (b+c) \cdot d(x_i,x_{i+1}) = k \cdot d(x_i,x_{i+1}),
    \end{equation} 
    where $0 < k = b+c < 1$. Thus 
    \[ d(x_i, x_{i+m}) \le k^{i+m} d(x_0,x_1) \to_{i,m \to \infty} 0.
    \]
    Therefore, $\{x_n\}_{n=0}^{\infty}$ is a Cauchy sequence. 
    By Proposition~\ref{eventuallyConst}, for some~$N$,
    $x_N$ is a fixed point of~$f$.

    To show the uniqueness of this fixed point, let $x'$ be
    a fixed point of~$f$ and suppose $x' \neq x_N$. Then
    \[ d(fx_N, fx') + a \cdot d(x',fx_N) \le 
       b \cdot \frac{[d(x_N, fx_N)]^2}{d(x_n,x')} + c \cdot d(x_N,x')
    \]
    or
    \[ d(x_N, x') + a \cdot d(x',x_N) \le 0 + c \cdot d(x_N,x')
    \]
    \[ (1+a) \cdot d(x_N,x') \le c \cdot d(x_N,x')
    \]
    \[ (1+a-c) \cdot d(x_N,x') \le 0.
    \]
    Since $c < 1$, the previous inequality contradicts the
    assumption that $x' \neq x_N$.
\end{proof}

\section{\cite{OAA}'s assertions}
Two assertions are presented as new theorems in~\cite{OAA}.
We show that one of these reduces to triviality,
and we show the other is not correctly proven.

\subsection{\cite{OAA}'s ``Theorem" 3.1}
We use the following.

\begin{definition}
    {\rm \cite{EgeKaraca-Ban}}
    Let $(X,d,\kappa)$ be a digital metric space and
    $f: X \to X$.
    If for some $\alpha \in [0,1)$ and all $x,y \in X$ we have
    \[ d(fx,fy) < \alpha d(x,y)
    \]
    then $f$ is a {\em digital contraction map}.
\end{definition}

The digital Banach contraction principle is the following.

\begin{thm}
\label{BanachPrinc}
    {\rm \cite{EgeKaraca-Ban}}; corrected proof~{\rm \cite{BxBad10}}
     Let $(X,d,\kappa)$ be a uniformly discrete digital metric space 
     and let $f: X \to X$ be a digital contraction map.
    Then~$f$ has a unique fixed point.
\end{thm}

The paper~\cite{OAA} states the following as its ``Theorem"~3.1.

\begin{assert}
    \label{OAAthm3.1}
    Let $(X,d,u)$ be a complete digital metric space.
    Let $G: X \to X$ be a mapping such that, for all $x,y \in X$ and for all non-negative $a,b,c$ with $a+b+c < 1$, \newline
    $(G_1)$ $d(Gx,Gy) \le a \cdot d(x,y)$ \newline
    $(G_2)$ $d(Gx,Gy) \le b \cdot [d(x,Gx) + d(y,Gy)]$ \newline
    $(G_3)$ $d(Gx,Gy) \le c \cdot [d(x,Gy) + d(y,Gx)]$ \newline
    and
    \[ d(Gx,Gy) < \beta \cdot \max
    \left \{ d(x,y), \frac{d(x,Gx)+ d(y,Gy)}{2}, \frac{d(x,Gy)+ d(y,Gx)}{2}
    \right \}.
    \]
   Let $P: X \to X$ be a mapping such that, for all $x,y \in X$, \newline
    $(P_1)$ $d(Px,Py) \le a \cdot d(x,y)$ \newline
    $(P_2)$ $d(Px,Py) \le b \cdot [d(x,Px) + d(y,Py)]$ \newline
    $(P_3)$ $d(Px,Py) \le c \cdot [d(x,Py) + d(y,Px)]$ \newline
    and
    \[ d(Px,Py) < \beta \cdot \max
    \left \{ d(x,y),~\frac{d(x,Px)+ d(y,Py)}{2},~\frac{d(x,Py)+ d(y,Px)}{2}
    \right \}.
    \] 
    Then, if $G$ and $P$ commute, they have a unique and common fixed
    point.
\end{assert}

Notice:
\begin{itemize}
    \item The digital Banach Fixed Point Principle (\cite{EgeKaraca-Ban}; corrected~ proof~\cite{BxBad10}) requires only a particular $a \in [0,1)$,
          not all $a,b,c \in [0,1)$, and does not require $(G_2)$ or $(G_3)$ to show
          the existence of a fixed point for a function~$G$ that satisfies~$(G_1)$.
          Similarly for the function~$P$. 
    \item No restriction on~$\beta$ is given in the statement of
    Assertion~\ref{OAAthm3.1}.
    \item Perhaps more importantly, in spite of the multipage argument offered as its
    proof in~\cite{OAA}, this proposition reduces to triviality, as we show below,
    despite omitting several of the hypotheses of~Assertion~\ref{OAAthm3.1}.
    \item Also, the argument offered as ``proof" has flaws including:
    \begin{itemize}
        \item an undefined notion, used in the second line of the argument, 
              that a point is ``contained" in another point; and
        \item an undefined notion, used in the fifth line of the argument, 
              that a distance is ``contained" in another distance.
    \end{itemize}
\end{itemize}

We modify Assertion~\ref{OAAthm3.1} to obtain the following Proposition~\ref{OAA3.1finite}, which shows how Assertion~\ref{OAAthm3.1}
reduces to triviality for a finite digital metric space.

By the {\em diameter} of a finite metric space $(X,d)$ we mean
\[ diam(X) = \max \{d(x,y) \mid x,y \in X \}.
\]

\begin{prop}
\label{OAA3.1finite}
    Let $(X,d,u)$ be a finite digital metric space.
    Let $G$ and~$P$ be self-maps on $X$ satisfying~($G_1$) and~($P_1$) of Assertion~\ref{OAAthm3.1}, respectively, for some coefficient~$a$ such that
    \[ 0 \le a < \frac{\min \{d(u,v) \mid u,v \in X, u \neq v \}}{diam(X)}.
    \] 
    Then each of $G, P$ has a unique fixed point, $u_G,u_P$, respectively.
    If~$G$ and~$P$ commute, then $G$ and $P$ are equal constant functions.
\end{prop}

\begin{proof}
    By ($G_1$) and ($P_1$), each of $G, P$ is a digital
    contraction map. By Theorem~\ref{BanachPrinc}, each has
    a unique fixed point, $u_G,u_P$, respectively.

    Then ($G_1$) implies for all distinct $x,y \in X$,
    \[ d(Gx,Gy) \le a \cdot d(x,y) < \frac{\min \{d(u,v) \mid u,v \in X, u \neq v \}}{diam(X)} \cdot d(x,y) \le
    \]
    \[ \min \{d(u,v) \mid u,v \in X, u \neq v \} \]
    so $d(Gx,Gy) = 0$, and similarly, $d(Px,Py) = 0$.
    Thus~$G$ and~$P$ are constant functions taking values
    $u_G,u_P$, respectively.
    
     If~$G$ and~$P$ commute, we have
     \[ u_G = G(P(u_G)) = P(G(u_G)) = u_P.
     \]
     Thus $G=P$.
\end{proof}

\subsection{\cite{OAA}'s ``Theorem" 3.2}
Assuming~\cite{OAA}'s ``$(G(Gx), G(Gy))$" is intended to be
``$d(G(Gx), G(Gy))$" and ``$(P(Px), P(Py))$" is intended to be
``$d(P(Px), P(Py))$", the following is a corrected statement 
of~\cite{OAA}'s ``Theorem"~3.2.

\begin{assert}
\label{OAA3.2}
    Let $(X,d,u)$ be a digital metric space. Let
    $G,P: X \to X$. Suppose for all
    $x,y \in X$ and non-negative constants $e,f,g,h,i$ such that
    \[ e+f+g+h+i < 1,
    \]
    we have
    \[ d(G(Gx),G(Gy)) \le e \cdot d(Gx,Gy) + f \cdot d(Gx, G(Gx)) + g \cdot d(Gy,G(Gy)) +
    \]
    \[ h \cdot d(Gx,G(Gy)) + i \cdot d(Gy, G(Gx))
    \]
    and
     \[ d(P(Px),P(Py)) \le e \cdot d(Px,Py) + f \cdot d(Px, P(Px)) + g \cdot d(Py,P(Py)) +
    \]
    \[ h \cdot d(Px,P(Py)) + i \cdot d(Py, P(Px)).
    \]
    Then, if $G$ and $P$ commute, they have a unique and common fixed
     point.   
\end{assert}

That Assertion~\ref{OAA3.2} is not correctly proven is shown as follows.
Errors in the argument offered as proof in~\cite{OAA} include:
\begin{itemize}
    \item In the first two lines of the ``proof" we find
    ``If $G(x) \subset P(x)$", an undefined notion. 
    \item On the 3rd line of the ``proof" we have the claim that
    \[ (G(Gx), G(Gy)) \le (P(Px), P(P(y)))
    \]
    which is unsupported even if rewritten as
     \[ d(G(Gx), G(Gy)) \le d(P(Px), P(P(y))).
    \]
    \item On the 2nd line of the ``proof's" second page, we find
    ``Since $G(x) \subset P(x)$" [as above, undefined] "it follows
    that $G(x) \le P(x)$" [also undefined since $X$ is not assumed
    a subset of~$\Z$].
    \item Later in the ``proof", $r$ is defined by
    \[r = \frac{e+f+h}{1-g-h}\]
    and it is claimed that $r < 1$, which is false for some
    choices of $e,f,g,h,i$.
\end{itemize}
%Other errors in the ``proof" follow.
We must conclude that Assertion~\ref{OAA3.2} is unproven.

\section{\cite{SalJah}'s assertion}
The paper~\cite{SalJah} claims the following as its ``Theorem"~3.1.

\begin{assert}
\label{SalJahAssert}
Let $(X,d)$ be a complete digital metric space. Let
$T: X \to X$ be a mapping. Let $k_1,k_2,k_3$ be nonnegative
constants such that
\begin{equation}
\label{SalJahAssume}
    k_1^2 + k_2^2 + k_3^2 < 1
\end{equation} 
and for all $a,b \in X$,
\[ d(Ta,Tb) \le \]
\begin{equation}
    \label{SalJahIneq}
    k_1^2 \cdot d(a,b) + k_2^2 \cdot [d(a,Ta) + d(b,Tb)] +
    k_3^2 \cdot \sqrt{d(a,b) \cdot \min \{ d(a,Ta), d(b,Tb)\}}.
\end{equation}
Then $T$ has a unique fixed point in~$X$.
\end{assert}

That Assertion~\ref{SalJahAssert} is false is shown by the
following example.

\begin{exl}
    Let $X = [0,1]_{\Z}$ and let $T: X \to X$ be given by
    $T(x) = 1 - x$. Notice $T$ has no fixed point. However, we
    can choose $k_1,k_2,k_3$ to satisfy the hypotheses of 
    Assertion~\ref{SalJahAssert}.
\end{exl}

\begin{proof}
    Take $d$ to be the usual metric for~$\Z$, $d(x,y) = |x-y|$.
    Let $k_1 = 0$, $k_2 = \sqrt{0.9}$, $k_3 = 0.3$.
    Then
    \[ k_1^2 + k_2^2 + k_3^2 = 0 + 0.9 + 0.09 = 0.99 < 1
    \]
    and~(\ref{SalJahIneq}) reduces to
        \begin{equation}
        \label{SalJahIneqReduced}
        d(Ta,Tb) \le 0.9 \cdot [d(a,Ta) + d(b,Tb)]  + 0.09 \cdot
       \sqrt{d(a,b) \cdot \min\{d(a,Ta), d(b,Tb)\}}.
    \end{equation}

    To show that~(\ref{SalJahIneqReduced}) holds, we
    consider the following cases.
    \begin{itemize}
        \item If $a = b$, the left side of~(\ref{SalJahIneqReduced})
        is 0, so~(\ref{SalJahIneqReduced}) holds in this case.
        \item If $a \neq b$, without loss of generality, $a=0$ and
        $b=1$. Then
        \[ d(Ta,Tb) = d(1,0) = 1 < 1.89 = 0.9 (2) + 0.09 \sqrt{1 \cdot \min\{1,1\}} = 
        \] 
        \[ 0.9 \cdot [d(a,Ta) + d(b,Tb)] + 0.09 \sqrt{d(a,b) \cdot \min\{d(a,Ta), d(b,Tb)\}},
        \]
        satisfying~(\ref{SalJahIneqReduced}).
    \end{itemize}
    Thus for all $a,b \in X$, (\ref{SalJahIneqReduced}) holds.
\end{proof}

\section{Further remarks}
We paraphrase~\cite{BxBad8}:
\begin{quote}
We have discussed several papers that seek to advance
fixed point assertions for digital metric spaces.
Many of these assertions are incorrect, incorrectly proven, 
or reduce to triviality; consequently, all of these
the authors, but also on the referees and editors who
approved their publication, and, perhaps, predatory journals
that accept payment for publication regardless of quality.
\end{quote}

\end{document}